\documentclass[11pt, a4paper,reqno]{amsart}
\usepackage[latin1]{inputenc}
\usepackage{amsfonts,amsmath,amssymb, amscd, graphicx}
\DeclareGraphicsExtensions{.jpg,.pdf,.mps,.eps,.png}

\usepackage[all]{xy}
\usepackage{psfrag}

\newtheorem{thm}{Theorem}[section]
\newtheorem{lem}[thm]{Lemma}
\newtheorem{defi}[thm]{Definition}
\newtheorem{prop}[thm]{Proposition}

\theoremstyle{definition}
\newtheorem{rem}[thm]{Remark}
\newtheorem{rems}[thm]{Remarks}

\newcommand{\Z}{\mathbb{Z}}

\newcommand{\R}{\mathbb{R}}

\newcommand{\ot}{\otimes}
\DeclareMathOperator{\B}{\mathcal{B}}
\DeclareMathOperator{\VB}{\mathcal{VB}}
\DeclareMathOperator{\W}{\mathcal{W}}
\DeclareMathOperator{\gen}{\mathcal{S}}

\DeclareMathOperator{\br}{br}
\DeclareMathOperator{\rb}{rb}
\DeclareMathOperator{\id}{id}
\DeclareMathOperator{\Aut}{Aut}

\numberwithin{equation}{section}

\title[Virtual braid groups of type~B and weak categorification]
{Virtual braid groups of type~B and \\ weak categorification}

\author{Anne-Laure Thiel}
\address{
Institut de Recherche Math\'e\-ma\-tique Avanc\'ee,
CNRS \& Universit\'e de Strasbourg,
7 rue Ren\'{e} Descartes, 67084 Strasbourg, France}

\email{thiel@math.unistra.fr}

\date{17.01.2011}

\begin{document}

\begin{abstract}
We define virtual braid groups of type~$B$ and construct a morphism from such a group to the group of isomorphism classes of some invertible complexes of bimodules up to homotopy.
\end{abstract}

\maketitle



\hspace{3cm}

\section*{Introduction}
Our aim in this paper is two-fold: we first define a virtual braid group of type~$B_n$ and
next construct a weak categorification of this group. 

Virtual knots and braids have been introduced by Kauffman in~\cite{KauVir};
they can be represented by planar diagrams that are like usual link or braid diagrams
with one extra type of crossings, called virtual crossings. 
Such crossings appear for instance when
one projects a generic link in a thickened surface onto a plane (see \cite{KaKa} or \cite{Ku}).

The virtual braids with $n$~strands form a group denoted~$\VB_n$,
which generalizes the usual Artin braid group~$\B_n$. 
The group~$\VB_n$ has a presentation with $2(n-1)$~generators
\[
\sigma_1, \ldots, \sigma_{n-1}, \zeta_1, \ldots, \zeta_{n-1} \, ,
\] 
where $\sigma_1, \ldots, \sigma_{n-1}$ satisfy the usual braid relations, 
$\zeta_1, \ldots, \zeta_{n-1}$ satisfy the standard defining relations of the symmetric group~$S_n$,
and the following ``mixed relations" are satisfied
\[
\sigma_i \zeta_j = \zeta_j \sigma_i , \quad  \mbox{if} \  |i-j| >1 ,
\]
and
\[
\sigma_i \zeta_{i+1} \zeta_i = \zeta_{i+1} \zeta_i \sigma_{i+1} , \quad \mbox{if} \  1 \leq i \leq n-2 .
\]

As is well known, 
the braid group~$\B_n$ can be generalized in the framework of Coxeter groups.
Recall that given any Coxeter system~$(\W, \gen)$ one defines a generalized braid group~$\B_{\W}$
by taking the same generators $s \in \gen$ and the same relations as for the Coxeter group~$\W$,
except the relations~$s^2 = 1$. When $\W$ is the symmetric group~$S_n$, i.e. of type~$A_{n-1}$ in the classification of Coxeter groups, then $\B_{\W} = \B_n$.

A natural question is:  can one similarly attach 
a generalized virtual group~$\VB_{\W}$ to any Coxeter system~$(\W, \gen)$? 
The idea for defining such a group~$\VB_{\W}$ would be, as in the type~$A$ case,
to use two copies of the generating set~$\gen$,
and to require that the first copy satisfies the relations defining~$\B_{\W}$, 
the second one satisfies the relations defining~$\W$, 
and some mixed relations involving generators from the two copies of $\gen$ are satisfied. 

The problem is to come up with an appropriate and meaningful set of mixed relations. 
In this paper, we do not solve the problem in the general case, but
we focus on the case of Coxeter groups of type~$B$.
We use a diagrammatic description of the generalized braid group~$\B_{\W}$
of type~$B$ due to tom Dieck~\cite{tD}; 
this description is in terms of symmetric braid diagrams. 
We define a generalized virtual braid group~$\VB_{B_n}$ of type~$B_n$
by considering symmetric virtual braid diagrams up to some natural equivalence.

In a second part we ``categorify" each newly-defined group~$\VB_{B_n}$.
Here we mean categorification in the weak sense of Rouquier, who actually proves a stronger version of this result in~\cite{R}. 
More precisely, here to any word~$w$ in the generators of~$\VB_{B_n}$
we associate a bounded cochain complex~$F(w)$ of bimodules such that 
if two words $w$ and $w'$ represent the same element of~$\VB_{B_n}$,
then the corresponding cochain complexes $F(w)$ and $F(w')$ are homotopy equivalent. This leads to a morphism from the group~$\VB_{B_n}$ to the group of isomorphism classes of invertible complexes up to homotopy, which is not injective. We describe a part of its kernel.
The bimodules used here have been introduced by Soergel, they have become important in the context of knot theory because they come up in the Khovanov-Rozansky link homology (see, e.g., \cite{Kh}).
Our categorification extends Rouquier's weak categorification of generalized braid groups
and our previous categorification~\cite{Th1} of the virtual braid group~$\VB_n$.

The paper is organized as follows. 
In Section~\ref{sec-VB} we recall the definition of the virtual braid groups~$\VB_n$
and of an invariant of virtual braids due to Manturov, which we see as 
a homomorphism of~$\VB_n$ into the automorphism group of a free group.
In Section~\ref{sec-symm} we recall the definition of the generalized braid group of type~$B_n$
and tom Dieck's graphical description in terms of symmetric braid diagrams.

We propose a definition of a generalized virtual braid group of type~$B_n$ in Section~\ref{sec-virsymm}.
We show that each of its elements can be represented by a symmetric virtual braid diagram, but we do not prove the injectivity of this representation.
Using Manturov's invariant, we show that certain relations do not hold in this group although they look natural.

Section~\ref{sec-Soergel} is devoted to the Soergel bimodules of type~$B_n$.
In Section~\ref{sec-cat} we associate a cochain complex of bimodules
to each generator of our virtual braid group of type~$B_n$, 
and we show that this leads to a categorification of this group in the weak sense of Rouquier.

\section{Virtual braids}\label{sec-VB}

We first recall the definition of virtual braid groups and of Manturov's invariant for virtual braids.

\subsection{The virtual braid groups}\label{sec-VB1}

Let $n$ be an integer~$\geq 2$. 
Following \cite{Man}, \cite{V}, we define the \emph{virtual braid group}~$\VB_n$
as the group generated by
$\sigma_1, \ldots, \sigma_{n-1}$ and $\zeta_1, \ldots, \zeta_{n-1}$,
and the following relations
\begin{equation}\label{vb1}
\sigma_i \sigma_j = \sigma_j \sigma_i  , \quad \mbox{if} \  |i-j| >1 ,
\end{equation}
\begin{equation}\label{vb2}
\sigma_i \sigma_{i+1} \sigma_i = \sigma_{i+1} \sigma_i \sigma_{i+1}  , \quad  \mbox{if} \  1 \leq i \leq n-2  ,
\end{equation}
\begin{equation}\label{vb3}
\zeta_i \zeta_j = \zeta_j \zeta_i , \quad \mbox{if} \ |i-j| >1  ,
\end{equation}
\begin{equation}\label{vb4}
\zeta_i \zeta_{i+1} \zeta_i = \zeta_{i+1} \zeta_i \zeta_{i+1} , \quad  \mbox{if} \  1 \leq i \leq n-2  ,
\end{equation}
\begin{equation}\label{vb5}
\zeta_{i}^{2} = 1  , \quad  \mbox{if} \  1 \leq i \leq n-1  ,
\end{equation}
\begin{equation}\label{vb6}
\sigma_i \zeta_j = \zeta_j \sigma_i , \quad  \mbox{if} \  |i-j| >1 ,
\end{equation}
\begin{equation}\label{vb7}
\sigma_i \zeta_{i+1} \zeta_i = \zeta_{i+1} \zeta_i \sigma_{i+1} , \quad  \mbox{if} \  1 \leq i \leq n-2 .
\end{equation}
Relations~\eqref{vb1} and~\eqref{vb2} are called \emph{braid relations}
and
Relations~\eqref{vb3}--\eqref{vb5} \emph{permutation relations}.
Relations~\eqref{vb6} and~\eqref{vb7} are called \emph{mixed relations}
because they involve both generators~$\sigma_i$ and~$\zeta_i$.

Elements of~$\VB_n$ can be represented by \emph{virtual braid diagrams with~$n$ strands}.
Such a diagram is a planar braid diagram 
with \emph{virtual crossings} in addition to the usual positive and negative crossings.

The generator $\sigma_i$ can be represented by the usual braid diagram 
with a single positive crossing between the $i$th and the $i+1$st strand; see Figure~\ref{fig:genpos}, 
whereas the generator~$\zeta_i$ is represented by the virtual braid diagram 
with a single virtual crossing between the $i$th and the $i+1$st strand; 
see Figure~\ref{fig:genvirtuel}.

\begin{figure}[htbp]
  \begin{center}
   \psfrag{1}{$\scriptstyle{1}$}
  \psfrag{i-1}{$\scriptstyle{i-1}$}
  \psfrag{i}{$\scriptstyle{i}$}
  \psfrag{i+1}{$\scriptstyle{i+1}$}
  \psfrag{i+2}{$\scriptstyle{i+2}$}
  \psfrag{m}{$\scriptstyle{n}$}
   \includegraphics[height=2.7cm]{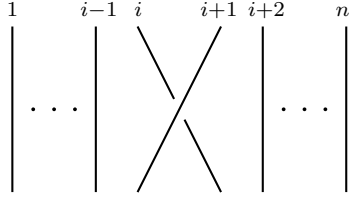}  
   \caption{The braid diagram $\sigma_i$}
   \label{fig:genpos}
  \end{center}
\end{figure}

\begin{figure}[htbp]
  \begin{center}
   \psfrag{1}{$\scriptstyle{1}$}
  \psfrag{i-1}{$\scriptstyle{i-1}$}
  \psfrag{i}{$\scriptstyle{i}$}
  \psfrag{i+1}{$\scriptstyle{i+1}$}
  \psfrag{i+2}{$\scriptstyle{i+2}$}
  \psfrag{m}{$\scriptstyle{n}$}
   \includegraphics[height=2.7cm]{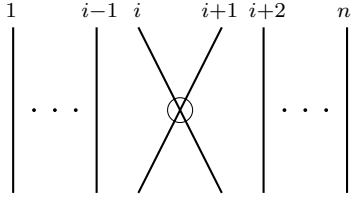}  
   \caption{The braid diagram $\zeta_i$}
   \label{fig:genvirtuel}
  \end{center}
\end{figure}

We use the following convention: if $D$ (resp.\ $D'$) is a virtual braid diagram representing
an element $\beta$ (resp.\ $\beta'$) of~$\VB_n$, then the product $\beta\beta'$
is represented by the diagram obtained by putting $D'$ on top of~$D$ and gluing
the lower endpoints of~$D'$ to the upper endpoints of~$D$.

Relations~\eqref{vb1}, \eqref{vb3},~\eqref{vb6} mean that we consider these planar diagrams
up to planar isotopy preserving the crossings. 
Relation~\eqref{vb2} illustrates the classical Reidemeister~III move
for ordinary braid diagrams.

Relations~\eqref{vb4} and~\eqref{vb5} mean that we consider the virtual braid diagrams 
up to the virtual Reidemeister~II--III moves depicted in Figure~\ref{fig:RV}.
Relation~\eqref{vb7} is a graphical transcription of the 
mixed Reidemeister move of Figure~\ref{fig:Rmixte}.
(Note that this mixed move involves one positive crossing and two virtual ones.)

\begin{figure}[h!]
  \begin{center}
   \includegraphics[height=2cm]{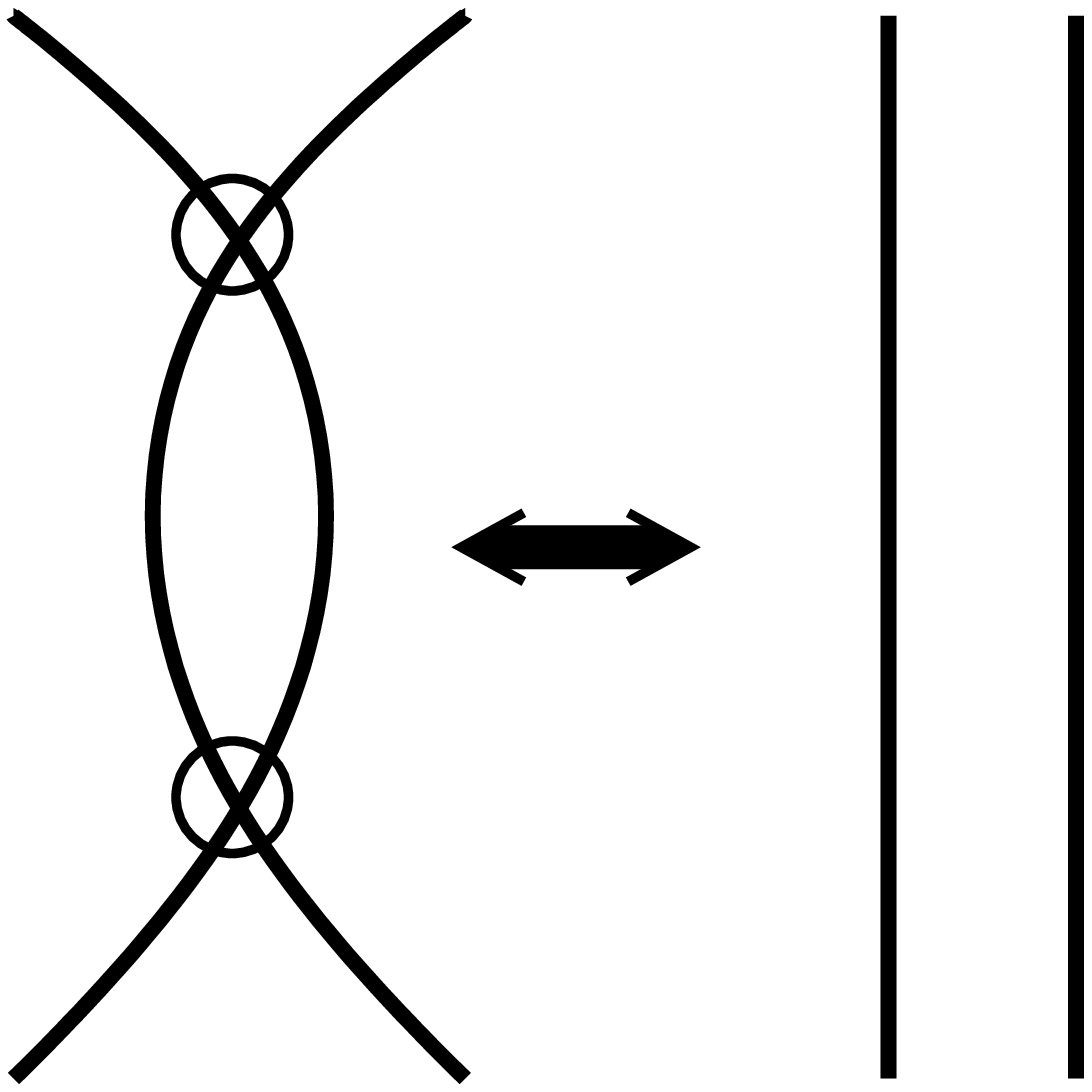}
   \hspace{2cm}
   \includegraphics[height=2cm]{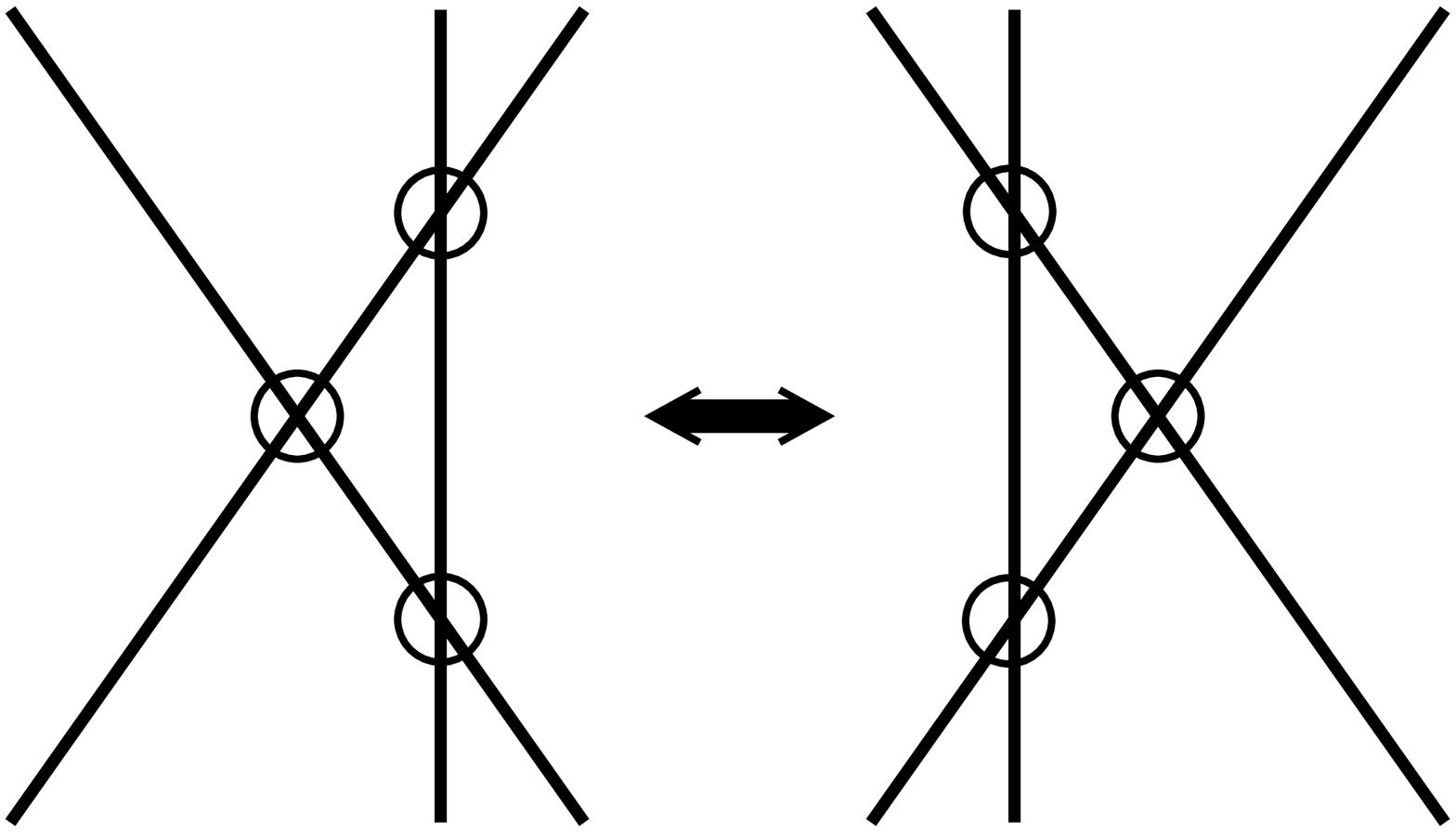}  
   \caption{Virtual Reidemeister moves}
   \label{fig:RV}
  \end{center}
\end{figure}

\begin{figure}[h!]
  \begin{center}
   \includegraphics[height=2cm]{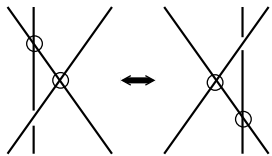}  
   \caption{The mixed Reidemeister move}
   \label{fig:Rmixte}
  \end{center}
\end{figure}

The braid group~$\B_n$ is obtained by dropping the generators $\zeta_1, \ldots, \zeta_{n-1}$. There is a natural homomorphism
$\B_n \to \VB_n$ obtained by considering each braid diagram as a virtual braid diagram
without virtual crossings.

\subsection{Manturov's invariant}\label{sec-MI}

Manturov~\cite{MaRec} constructed an invariant of virtual braids, which he conjectured to be complete. 
Since we will be using it in Proposition~\ref{bizarre}, we have to recall its definition.
We use the version that appeared in the review~\cite{Iz}.

In this version, Manturov's invariant can be seen as a group homomorphism 
$f : \VB_n \to \Aut(F_{n+1})$, where $F_{n+1}$ is the free group on $n+1$
generators $a_1,\ldots,a_n, t$.
The homomorphism~$f$ is defined by the following formulas:
\begin{align*}
f(\sigma_i)(a_j) & = 
\begin{cases}
a_{j+1} & \text{if $j=i$,}\\
a_j^{-1} a_{j-1} a_j & \text{if $j=i+1$,}\\
a_j & \text{otherwise},
\end{cases}\\
f(\zeta_i)(a_j) & =
\begin{cases}
t a_{j+1} t^{-1} & \text{if $j=i$,}\\
t^{-1} a_{j-1} t & \text{if $j=i+1$,}\\
a_j & \text{otherwise}, \\
\end{cases}
\end{align*}
and $f(\sigma_i)(t) = f(\zeta_i)(t) = t$ for all $i= 1, \ldots, n-1$.

Observe that this invariant specializes when $t=1$ to an invariant constructed in~\cite{FRR} for welded braids; see also~\cite{V}.

\section{Braid groups of type~$B$ and symmetric braid diagrams}\label{sec-symm}


Let $n$ be a positive integer. 
Consider the Coxeter group associated to the Dynkin diagram of type~$B_n$;
see~\cite{Hu}.

We denote the associated generalized braid group by~$\B_{B_n}$.
It has a presentation with $n$~generators $s_0, s_1, \ldots, s_{n-1}$, 
and three families of relations
\begin{equation}\label{relB1}
s_is_j=s_js_i , \quad  \mbox{if} \   |i-j|>1 ,
\end{equation}
\begin{equation}\label{relB2}
s_is_{i+1}s_i = s_{i+1}s_is_{i+1} , \quad  \mbox{if} \  1 \leq i \leq n-2  ,
\end{equation}
\begin{equation}
s_0s_1s_0s_1 = s_1s_0s_1s_0  .
\label{relB3}
\end{equation}

In~\cite{tD} tom Dieck gave a diagrammatic description of~$\B_{B_n}$ 
in terms of symmetric braid diagrams with $2n$~strands.

Fix a vertical line $\{0\} \times \R$ in the plane~$\R^2$. Consider the reflection in this vertical line. It acts on the underlying graph of a braid diagram. It lifts to a reflection on braid diagrams by imposing that the sign of the crossings is preserved. (Observe that if one identifies the plane $\R^2$ with $\{0\}\times\R^2 \subset \R^3$, this reflection can be viewed as a rotation of angle $\pi$ around the vertical axis $\{0\}^2 \times \R$.)
A \emph{symmetric braid diagram with $2n$ strands} is a planar braid diagram with $2n$ strands
that is invariant under this reflection.
To make things precise, we assume that the upper (resp.\ lower) endpoints 
of each symmetric braid diagram with $2n$ strands are the points
$\{-n \dots ,-1, 1,\dots,n\} \times \{1\}$
(resp.\ the points $\{-n \dots ,-1, 1,\dots,n\} \times \{0\}$).
We consider the symmetric braid diagrams with $2n$~strands up to planar isotopy
preserving the crossings,
and up to the classical Reidemeister~II--III moves, as depicted in Figure~\ref{fig:mouv Reidemeister}.
We stress the fact that neither the isotopies nor the Reidemeister moves
have to be preserved under the reflection.

\begin{figure}[htbp]
   \begin{center}
   \includegraphics[height=2cm]{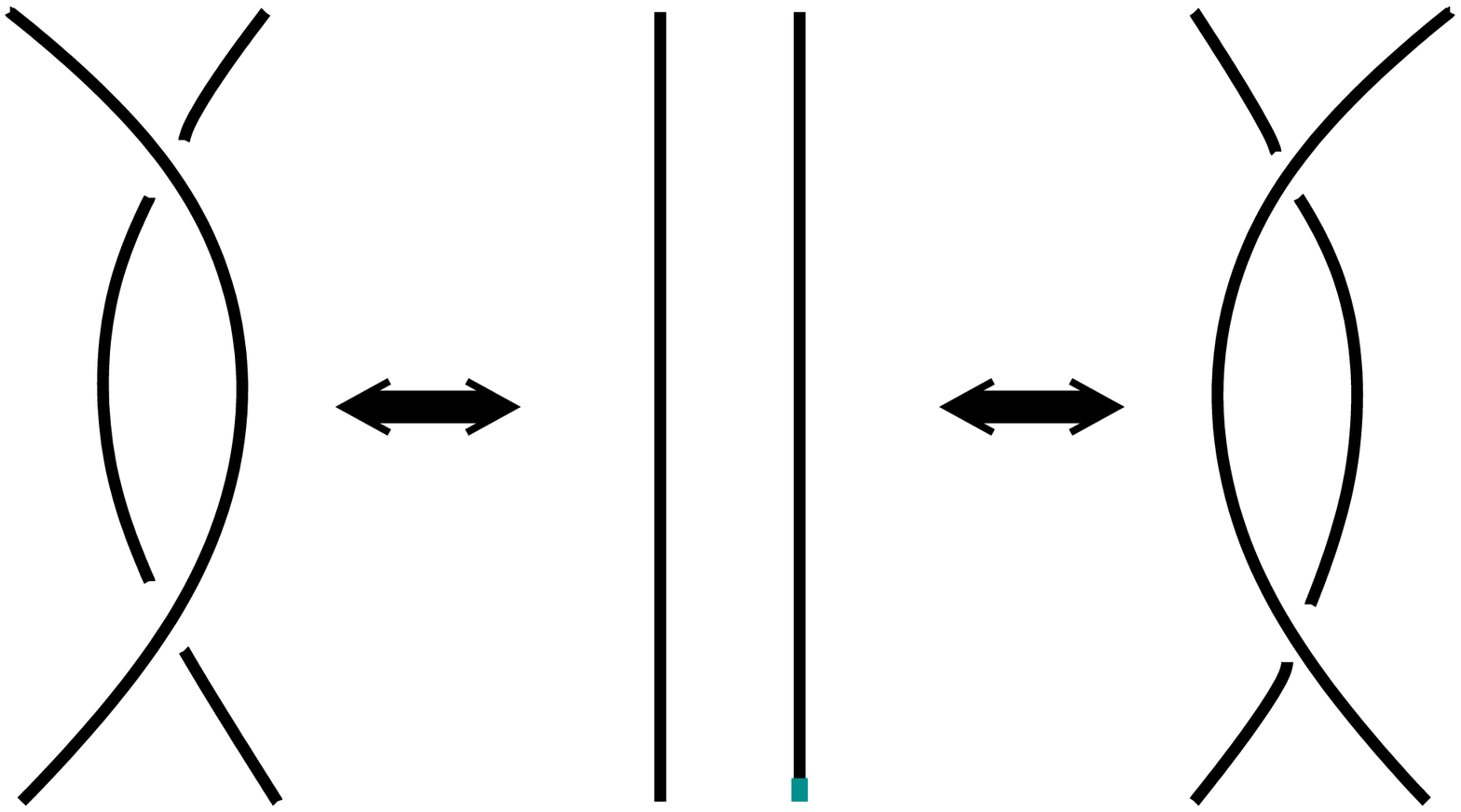}
   \hspace{2cm}
   \includegraphics[height=2cm]{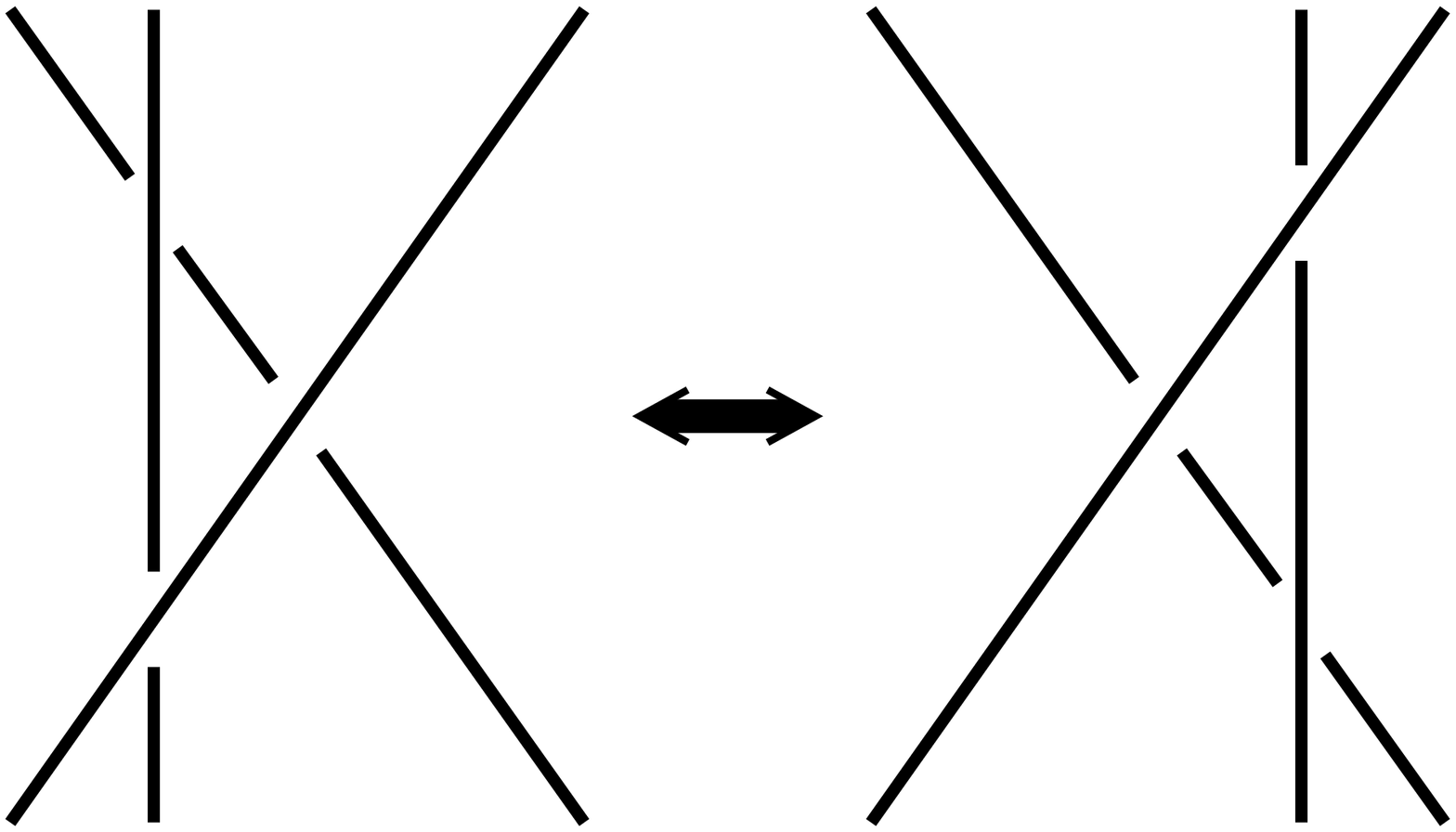}  
   \caption{Classical Reidemeister~II--III moves}
   \label{fig:mouv Reidemeister}
  \end{center}
\end{figure}

The equivalence classes of symmetric braid diagrams with $2n$ strands form a group,
which tom Dieck~\cite{tD} proved to be isomorphic to the generalized braid group~$\B_{B_n}$
of type~$B_n$.
In this isomorphism, the generator~$s_0$ is represented by the symmetric braid diagram 
with one positive crossing, as in Figure~\ref{fig:elem sym braid 0}, 
and each remaining generator $s_1, \ldots, s_{n-1}$ by the symmetric braid diagram 
with two symmetric positive crossings, as in Figure~\ref{fig:elem sym braid i}. 
From now on, we will identify each generator $s_i$ of~$\B_{B_n}$ with 
the corresponding symmetric braid. 

\begin{figure}[h!]
   \begin{center}
   \psfrag{1}{$\scriptstyle{1}$}
   \psfrag{2}{$\scriptstyle{2}$}
   \psfrag{m}{$\scriptstyle{n}$}
   \psfrag{-1}{$\scriptstyle{-1}$}
   \psfrag{-2}{$\scriptstyle{-2}$}
   \psfrag{-m}{$\scriptstyle{-n}$}
   \includegraphics[height=2.7cm]{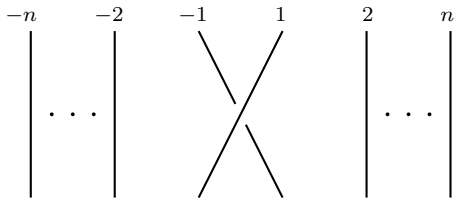}  
   \caption{The symmetric braid diagram~$s_0$}
   \label{fig:elem sym braid 0}
   \end{center}
\end{figure}

\begin{figure}[htbp]
  \begin{center}
  \psfrag{1}{$\scriptstyle{1}$}
  \psfrag{i-1}{$\scriptstyle{i-1}$}
  \psfrag{i}{$\scriptstyle{i}$}
  \psfrag{i+1}{$\scriptstyle{i+1}$}
  \psfrag{i+2}{$\scriptstyle{i+2}$}
  \psfrag{m}{$\scriptstyle{n}$}
  \psfrag{-1}{$\scriptstyle{-1}$}
  \psfrag{-i+1}{$\scriptstyle{-i+1}$}
  \psfrag{-i}{$\scriptstyle{-i}$}
  \psfrag{-i-1}{$\scriptstyle{-i-1}$}
  \psfrag{-i-2}{$\scriptstyle{-i-2}$}
  \psfrag{-m}{$\scriptstyle{-n}$}
   \includegraphics[height=2.7cm]{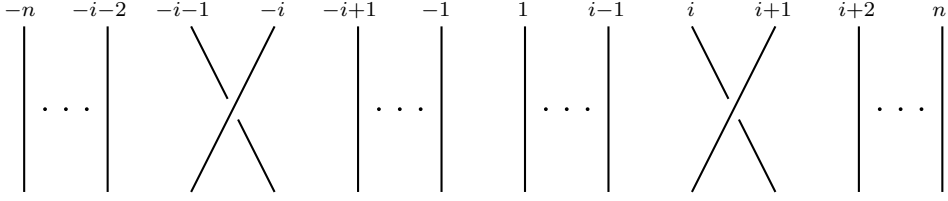}  
   \caption{The symmetric braid diagram~$s_i$ ($i>0$)}
   \label{fig:elem sym braid i}
  \end{center}
\end{figure}

In terms of symmetric braid diagrams, we see that
Relation~\eqref{relB1} holds because the symmetric braid diagrams corresponding to each
term of the relation are isotopic.
Similarly, the equality in Relation~\eqref{relB2} can be proven diagramatically using
Reidemeister~III move.
As for Relation~\eqref{relB3}, it can be proven as shown in Figure~\ref{fig:relquad}.
(Note that four Reidemeister~III moves have been used.)

\begin{figure}[h!]
   \begin{center}
   \includegraphics[height=3cm]{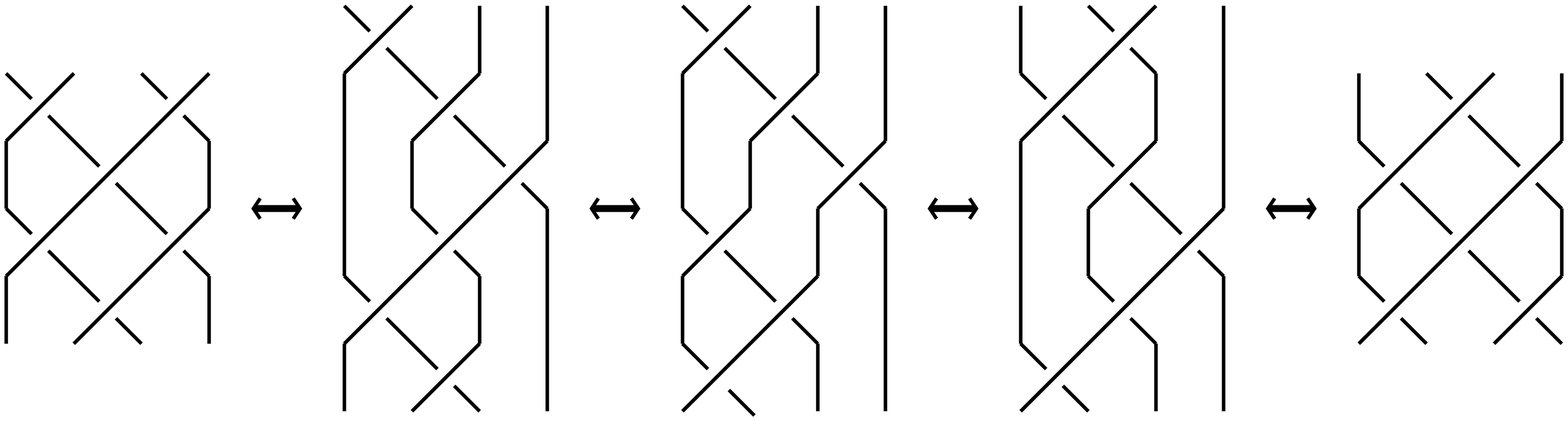}
   \caption{Proving Relation \eqref{relB3}}
   \label{fig:relquad}
   \end{center}
\end{figure}

Forgetting the symmetry condition yields an embedding of~$\B_{B_n}$ 
into the group~$\B_{2n}$ of usual braids with $2n$~strands. 
Before we give a precise for- mula for this embedding, let us 
shift the indices of the $2n-1$ generators~$\sigma_i$ of~$\B_{2n}$ by~$-n$; 
in this way the indexing set for the generators becomes the set $\{-n +1, \ldots, n-1\}$.
The re-indexed generators satisfy the same braid relations~\eqref{vb1} and~\eqref{vb2}.
We then define an embedding of~$\B_{B_n}$ into the group~$\B_{2n}$ by sending~$s_0$ to~$\sigma_0$, 
and each remaining~$s_i$ to~$\sigma_{-i}\sigma_i = \sigma_i \sigma_{-i}$ ($i=1, \dots , n-1$).

\section{Symmetric virtual braids}\label{sec-virsymm}

We now define a group~$\VB_{B_n}$, 
which will be our generalized virtual braid group of type~$B_n$.

\begin{defi}
The group~$\VB_{B_n}$ has the following presentation:
it is generated by $2n$ generators $s_0, s_1, \ldots, s_{n-1}$ and $z_0, z_1, \ldots, z_{n-1}$,
where $s_0, s_1, \ldots$, $s_{n-1}$ satisfy 
Relations~\eqref{relB1}, \eqref{relB2}, \eqref{relB3}, 
$z_1, \ldots, z_{n-1}$ satisfy the relations
\begin{equation}\label{relWB1}
z_i z_j = z_j z_i  ,\quad  \mbox{if} \  |i-j|>1 ,
\end{equation}
\begin{equation}\label{relWB2}
z_i z_{i+1}z_i = z_{i+1}z_i z_{i+1}  , \quad  \mbox{if} \  1 \leq i \leq n-2 ,
\end{equation}
\begin{equation}\label{relWB3}
z_0 z_1 z_0 z_1 = z_1 z_0 z_1 z_0  ,
\end{equation}
\begin{equation}\label{relWB0}
z_{i}^{2} = 1   ,\quad  \mbox{if} \  0 \leq i \leq n-1  ,
\end{equation}
and the following ``mixed relations" are satisfied
\begin{equation}\label{relmixB1}
s_i z_j = z_j s_i  ,\quad  \mbox{if} \   | i-j | > 1  ,
\end{equation}
\begin{equation}\label{relmixB2}
s_i z_{i+1} z_i = z_{i+1} z_i s_{i+1}  ,\quad  \mbox{if} \  1 \leq i \leq n-2  ,
\end{equation}
\begin{equation}\label{relmixB3}
s_0z_1z_0z_1 = z_1z_0z_1s_0  ,
\end{equation}
\begin{equation}\label{relmixB4}
z_0s_1z_0z_1 = z_1z_0s_1z_0  ,
\end{equation}
\begin{equation}\label{relmixB5}
s_0z_1s_0z_1 = z_1s_0z_1s_0  .
\end{equation}
\end{defi}

By analogy with tom Dieck's graphical description, we want to represent elements of~$\VB_{B_n}$
by \emph{symmetric virtual braid diagrams with $2n$ strands}.
These are planar virtual braid diagrams with $2n$ strands, as defined in Section~\ref{sec-VB1},
that are symmetric under the reflection in the vertical line $\{0\} \times \R$.
The reflection is supposed to preserve the virtual crossings as well as 
the positive (resp.\ the negative) crossings.
We consider the symmetric virtual braid diagrams with $2n$~strands up to planar isotopy
preserving the crossings,
up to the classical Reidemeister~II--III moves of Figure~\ref{fig:mouv Reidemeister}, 
up to the virtual Reidemeister~II--III moves of Figure~\ref{fig:RV}, 
and up to the mixed Reidemeister move of Figure~\ref{fig:Rmixte}.

We represent the generators $s_0, s_1, \ldots, s_{n-1}$ by the symmetric braid diagrams
of Section~\ref{sec-symm}. 
The generator~$z_0$ is sent to the symmetric virtual braid diagram 
with a single virtual crossing, as in Figure~\ref{fig:elem virt sym braid 0}.
Each remaining generator~$z_1, \ldots, z_{n-1}$ is sent to a symmetric virtual braid diagram 
with two symmetric virtual crossings as in Figure~\ref{fig:elem virt sym braid i}.

\begin{figure}[h!]
   \begin{center}
   \psfrag{1}{$\scriptstyle{1}$}
   \psfrag{2}{$\scriptstyle{2}$}
   \psfrag{m}{$\scriptstyle{n}$}
   \psfrag{-1}{$\scriptstyle{-1}$}
   \psfrag{-2}{$\scriptstyle{-2}$}
   \psfrag{-m}{$\scriptstyle{-n}$}
   \includegraphics[height=2.7cm]{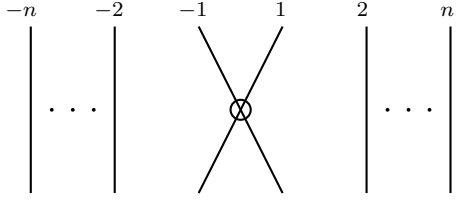}  
   \caption{The symmetric virtual braid diagram~$z_0$}
   \label{fig:elem virt sym braid 0}
   \end{center}
\end{figure}

\begin{figure}[htbp]
  \begin{center}
  \psfrag{1}{$\scriptstyle{1}$}
  \psfrag{i-1}{$\scriptstyle{i-1}$}
  \psfrag{i}{$\scriptstyle{i}$}
  \psfrag{i+1}{$\scriptstyle{i+1}$}
  \psfrag{i+2}{$\scriptstyle{i+2}$}
  \psfrag{m}{$\scriptstyle{n}$}
  \psfrag{-1}{$\scriptstyle{-1}$}
  \psfrag{-i+1}{$\scriptstyle{-i+1}$}
  \psfrag{-i}{$\scriptstyle{-i}$}
  \psfrag{-i-1}{$\scriptstyle{-i-1}$}
  \psfrag{-i-2}{$\scriptstyle{-i-2}$}
  \psfrag{-m}{$\scriptstyle{-n}$}
   \includegraphics[height=2.7cm]{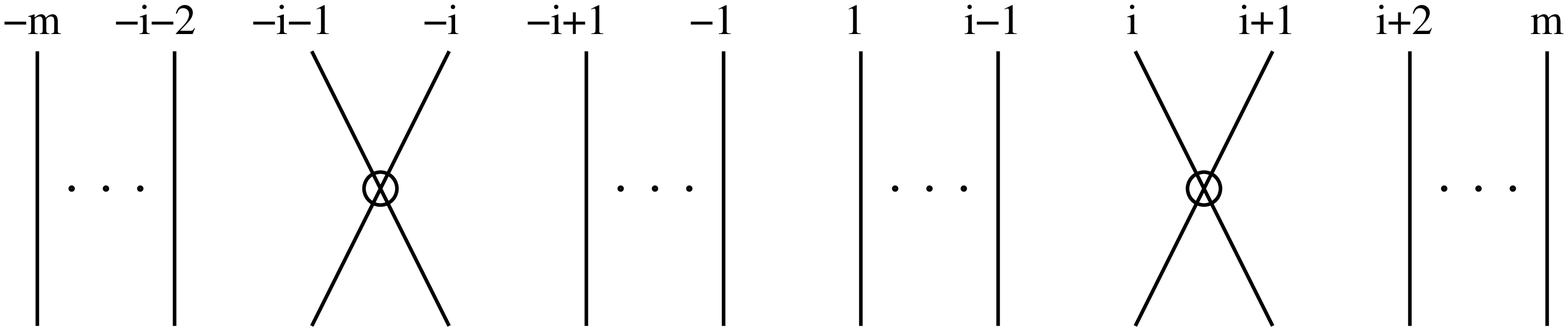}  
   \caption{The symmetric virtual braid diagram~$z_i$ ($i>0$)}
   \label{fig:elem virt sym braid i}
  \end{center}
\end{figure}

\begin{prop}
Considering symmetric braid diagrams as elements of $\VB_{2n}$ yields a well-defined group homomorphism $j : \VB_{B_n} \to \VB_{2n}$ defined by
\begin{equation}\label{id0}
j(s_0) = \sigma_0  , \qquad j(z_0) = \zeta_0  ,
\end{equation}
and
\begin{equation}\label{idi}
j(s_i) = \sigma_{-i} \sigma_i  = \sigma_i \sigma_{-i} , \qquad j(z_i) = \zeta_{-i} \zeta_i  = \zeta_i \zeta_{-i}
\end{equation}
for $1 \leq i \leq n-1$.
\end{prop}
Here again as in Section~\ref{sec-symm}, we have shifted the index of the generators of
the virtual braid group~$\VB_{2n}$ by~$-n$.

This shows that the relations defining the virtual braid group~$\VB_{B_n}$ of type~$B$
are adequate. 
\begin{rem} We do not claim that~$j: \VB_{B_n} \to \VB_{2n}$ is a monomorphism but we conjecture that it is.
Should this hold, then we could claim that the defining relations are sufficient
to define~$\VB_{B_n}$ as a subgroup of~$\VB_{2n}$.
\end{rem}

\begin{proof}
Using the relations defining~$\VB_{2n}$ (see Section~\ref{sec-VB1}),
we now check that $j(s_0), j(s_1), \ldots, j(s_{n-1})$,
$j(z_0), j(z_1), \ldots$, $j(z_{n-1})$ satisfy the defining relations of~$\VB_{B_n}$.

Relations~\eqref{relB1}--\eqref{relB3} and \eqref{relWB1}--\eqref{relmixB2} are obviously satisfied.

A graphical proof of Relation~\eqref{relmixB3} is given in Figure~\ref{fig:relquadmix1},
where we use the virtual Reidemeister~III move of Figure~\ref{fig:RV} twice and the mixed Reidemeister move
of Figure~\ref{fig:Rmixte} twice.

\begin{figure}[h!]
   \begin{center}
   \includegraphics[height=3cm]{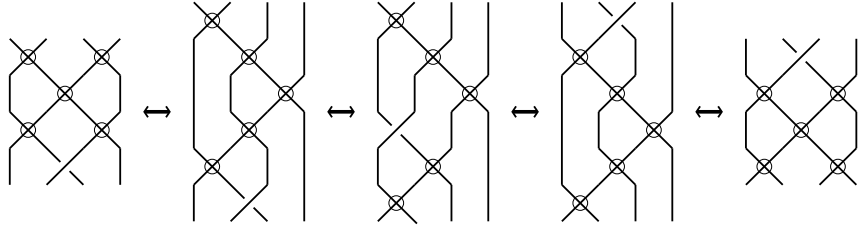}
   \caption{A graphical proof of Relation~\eqref{relmixB3}}
   \label{fig:relquadmix1}
   \end{center}
\end{figure}

We similarly obtain Relation~\eqref{relmixB4} using four mixed Reidemeister moves, 
as in Figure~\ref{fig:relquadmix2}.

\begin{figure}[h!]
   \begin{center}
   \includegraphics[height=3cm]{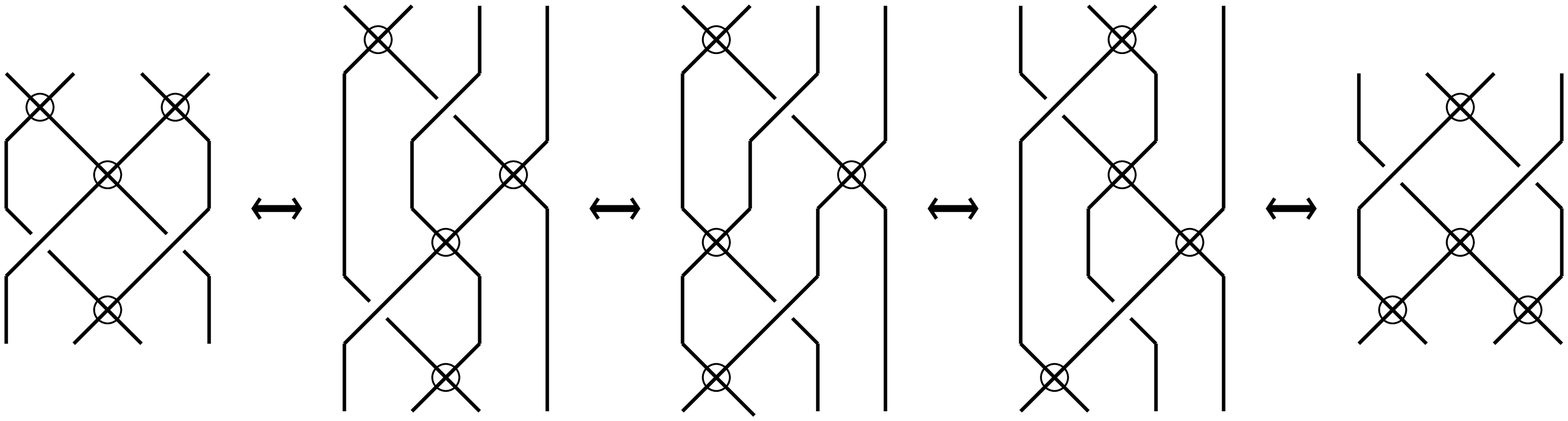}
   \caption{A graphical proof of Relation~\eqref{relmixB4}}
   \label{fig:relquadmix2}
   \end{center}
\end{figure}

Finally, a graphical proof of Relation~\eqref{relmixB5} is given in Figure~\ref{fig:relquadmix3}.

\begin{figure}[h!]
   \begin{center}
   \includegraphics[height=3cm]{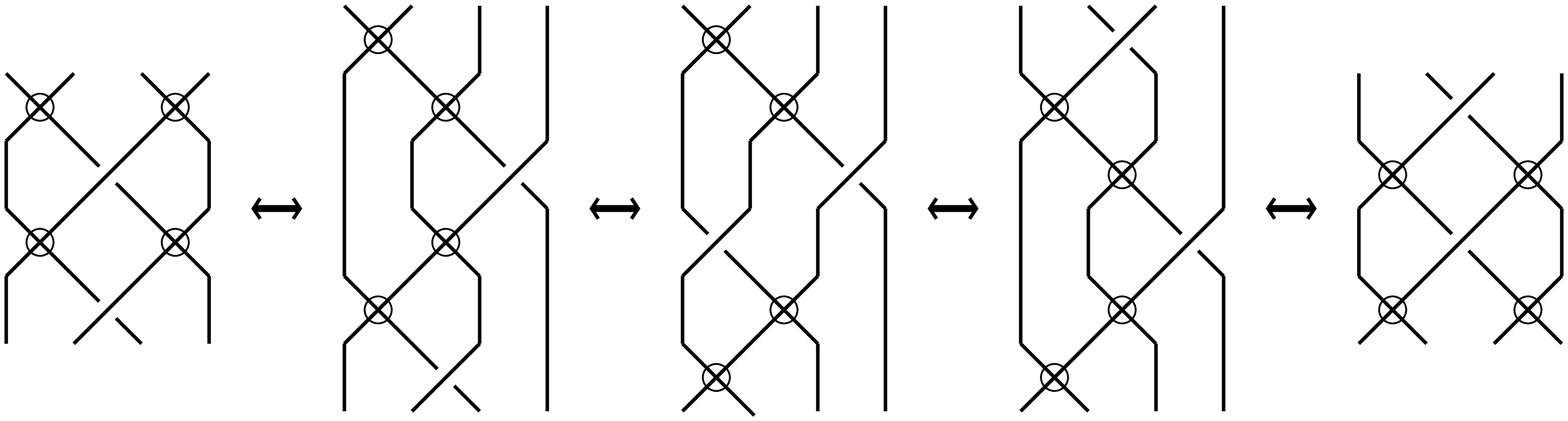}
   \caption{A graphical proof of Relation~\eqref{relmixB5}}
   \label{fig:relquadmix3}
   \end{center}
\end{figure}
\end{proof}

The inequalities recorded in the following proposition will come up in Remarks~\ref{pasbizarre}.

\begin{prop}\label{bizarre}

(a) In~$\VB_n$ we have 
\begin{equation*}
\sigma_i \zeta_i \neq \zeta_i \sigma_i 
\quad \text{and} \quad
\zeta_{i} \sigma_{i+1} \sigma_{i} \neq \sigma_{i+1} \sigma_{i} \zeta_{i+1} .
\end{equation*}

(b) In~$\VB_{B_n}$ we have
\begin{equation*}\label{eq-fausse}
z_0 s_1 z_0 s_1 \neq s_1 z_0 s_1 z_0 .
\end{equation*}
\end{prop}

\begin{proof}
(a) For each relation we will use Manturov's invariant~$f$ (see Section~\ref{sec-MI}) to show that both sides of the relation do not have the same invariant. This will imply the desired inequality in~$\VB_n$. For the leftmost relation, we have
$$ f\left( \zeta_i \sigma_i \right) (a_i) = t^{-1} a_i t \quad \mbox{and} \quad  f\left(  \sigma_i \zeta_i\right) (a_i) = t a_{i+1}^{-1} a_i a_{i+1} t^{-1}, $$
which are different.

For the righmost relation, we have
\begin{equation*}
f(\zeta_{i} \sigma_{i+1} \sigma_{i}):
\left\{
\begin{array}{ccll}
t & \mapsto & t  , & \\
a_i & \mapsto & a_{i+2} , &  \\
a_{i+1} & \mapsto & a_{i+2}^{-1}t a_{i+1} t^{-1}a_{i+2} , \\
a_{i+2} & \mapsto &  a_{i+2}^{-1}t^{-1} a_{i}t a_{i+2} , \\
a_j & \mapsto & a_j  & \mbox{if} \  j \neq i, i+1,i+2
\end{array}
\right.
\end{equation*}
and
\begin{equation*}
f(\sigma_{i+1} \sigma_{i} \zeta_{i+1}):
\left\{
\begin{array}{ccll}
t & \mapsto & t  , & \\
a_i & \mapsto & a_{i+2} , &  \\
a_{i+1} & \mapsto & t a_{i+2}^{-1} a_{i+1} a_{i+2} t^{-1}, \\
a_{i+2} & \mapsto & t^{-1} a_{i+2}^{-1} a_{i} a_{i+2} t, \\
a_j & \mapsto & a_j  & \mbox{if} \  j \neq i, i+1,i+2.
\end{array}
\right.
\end{equation*}
So $f(\zeta_{i} \sigma_{i+1} \sigma_{i})$ and $f(\sigma_{i+1} \sigma_{i} \zeta_{i+1})$ are not equal in $\Aut \left( F_{n+1} \right)$.

(b) We want to compare the two elements $z_0 s_1 z_0 s_1$ and $s_1 z_0 s_1 z_0$ of $\VB_{B_n}$. Their image in $\VB_{2n}$ under the homomorphism $j$ are respectively the elements $\zeta_0 \sigma_{-1} \sigma_1 \zeta_0 \sigma_{-1} \sigma_1 $ and $\sigma_{-1} \sigma_1 \zeta_0 \sigma_{-1} \sigma_1 \zeta_0 $. To them, we apply Manturov's invariant, the group homomorphism $f: \VB_{2n} \to \Aut(F_{2n+1})$,
where $F_{2n+1}$ is the free group on the generators $a_{-n+1}, \ldots, a_{-1}, a_0, a_1, \ldots, a_n,t$, which is given by the same formulas as in Section~\ref{sec-MI}.

A simple computation shows that
\[
f(\zeta_0 \sigma_{-1} \sigma_{1}\zeta_0 \sigma_{-1}\sigma_1)(a_2)
= a_2^{-1} t^{-1} a_0^{-1} t a_2 a_{1}^{-1} t^{-1} a_{-1} t a_1 a_2^{-1} t^{-1} a_0 t a_2
\]
and
\[
f(\sigma_{-1}\sigma_{1}\zeta_0 \sigma_{-1}\sigma_1\zeta_0)(a_2)
= a_2^{-1} a_1 ^{-1} a_2 t^{-1} a_{0}^{-1} a_{-1} a_0 t a_2^{-1} a_1 a_2 \, ,
\]
which are different.
\end{proof}

\begin{rems}\hfill\\
$\bullet$ The inequalities of Proposition~\ref{bizarre}\,(a) imply that, in $\VB_{B_n}$ we have
\begin{equation*}
s_i z_i \neq z_i s_i \quad \text{if} \  0 \leq i \leq n-1
\end{equation*}
and
\begin{equation*}
z_{i} s_{i+1} s_{i} \neq s_{i+1} s_{i} z_{i+1}  \quad \text{if} \  1 \leq i \leq n-1 .
\end{equation*}

\noindent $\bullet$ The second inequality of Proposition~\ref{bizarre}\,(a) means that the 
mixed Reidemeister move of Figure~\ref{fig:R3interdit}, involving one virtual crossing and two positive crossings, is not allowed. Specializing at $t=1$ one observes that $f(\zeta_{i} \sigma_{i+1} \sigma_{i})$ and $f(\sigma_{i+1} \sigma_{i} \zeta_{i+1})$ become equal in $\Aut \left( F_n \right)$. Adding the relation $\zeta_{i} \sigma_{i+1} \sigma_{i} = \sigma_{i+1} \sigma_{i} \zeta_{i+1}$ to the presentation of $\VB_n$, one precisely obtain a presentation of the group of welded braids with $n$ strands defined in~\cite{FRR}.
\end{rems}

\begin{figure}[htbp]
  \begin{center}
   \includegraphics[height=2cm]{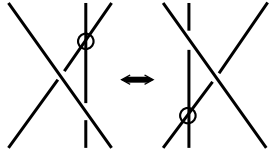}  
   \caption{A forbidden Reidemeister move}
   \label{fig:R3interdit}
  \end{center}
\end{figure}

\section{Soergel bimodules}\label{sec-Soergel}

\subsection{Generalities}

We consider some bimodules which have been introduced by Soergel \cite{S1}, \cite{S2}, \cite{S3} in the general context of Coxeter groups.

Let $(\W,\gen)$ be a Coxeter system with generating set $\gen = \{ s_0, \dots , s_{n-1}\}$ of cardinality~$n$. 
As usual, we denote the order of $st$, where $s,t \in \gen$, by~$m(s,t)$ and we set 
\[
\langle s_i, s_j \rangle = - \cos \frac{\pi}{m(s_i,s_j)}  \in \R \, .
\]

Let $R = \R[X_0, \ldots,X_{n-1}]$ be a real polynomial algebra in $n$~indeterminates $X_0, \ldots,X_{n-1}$.


For $i = 0, 1, \ldots, n-1$ we define an algebra automorphism~$\alpha_i$ of~$R$ by
\begin{equation*}
\alpha_i (X_j) = X_j -2 \, \langle s_i,s_j \rangle \, X_i
\end{equation*}
for all $j = 0, 1, \ldots, n-1$. It is easy to check that $\alpha_i$ is an involution
and that there is a unique action of~$\W$ on~$R$ by algebra automorphisms such that
\[
s_i (P) = \alpha_i(P)
\]
for all $i = 0, 1, \ldots, n-1$ and $P \in R$.

For any $w \in \W$ we consider the following objects:

\begin{itemize}

\item[(a)] the subalgebra~$R^w$ of elements of~$R$ fixed by~$w$;

\item[(b)] the $R$-bimodule $B_w = R \ot_{R^w} R$;

\item[(c)] the $R$-bimodule~$R_w$, which coincides with~$R$ as a left $R$-module whereas $a \in R$ acts on~$R_w$ on the right by multiplication by~$w(a)$.
\end{itemize}

\begin{rems} \hfill\\
$\bullet$ To avoid confusion, let us point out that:
\begin{itemize}
\item[-] among these bimodules, the only ones which are commonly called Soergel bimodules are the bimodules $B_s$ for $s \in \gen$,
\item[-] for an arbitrary $w \in \W$, Soergel uses the notation $B_w$ for different bimodules than us (some indecomposables) that we will not consider here.
\end{itemize}

\noindent $\bullet$ The bimodules $R_w$ already appeared in Soergel's work and they turned out to be important in the study of Soergel bimodules.
\end{rems}

We introduce a grading on the algebra $R$ by setting 
$\deg (X_k)=2$ for all $k=0,\dots,n-1$. It induces a grading on the bimodules we consider.
We will indicate a shift of the grading by curly brackets: if $M=\underset{i\in 
\Z}{\bigoplus}M_i$ is a $\Z$--graded bimodule and $p$ an integer, then the 
$\Z$--graded bimodule $M\{p\}$ is defined by $M\{p\}_i = M_{i-p}$ for all $i 
\in \Z$.

In the sequel we will systematically identify any bimodule $R_{w}\ot_R B_{w '}$ (resp. $ B_{w } \ot_R B_{w'}$) with $R_{w} \ot_{R^{w '}} R$ (resp. $R \ot_{R^{w }} R \ot_{R^{w'}} R$) since
$$R_{w}\ot_R B_{w '} =  R_{w} \ot_R R \ot_{R^{w '}} R \cong R_{w} \ot_{R^{w '}} R $$
and
$$ B_{w } \ot_R B_{w'} = R \ot_{R^{w }} R \ot_R R \ot_{R^{w'}} R \cong R \ot_{R^{w }} R \ot_{R^{w'}} R.$$

%

The following lemma is straightforward. It will be used repeatedly in Section~\ref{sec-ourcat}.

\begin{lem}\label{iso bimod} 
For all $w,w' \in \W$, there are isomorphisms of $R$-bimodules
\[
R_{w} \ot_R R_{w'} \overset{\cong}{\longrightarrow} R_{w w'}
\]
sending $a \ot b$ to $a w(b)$, and  
\[
R_{w}\ot_R B_{w '} \overset{\cong}{\longrightarrow}  B_{w w' w^{-1}} \ot_R R_{w}
\]
sending $a \ot b$ to $a \ot w(b)$ ($a,b \in R$).
\end{lem}

\subsection{The type $B_n$}\label{inv}

When the Coxeter system~$(\W,\gen)$ is of type~$B_n$, then 
\[
\begin{array}{rcccll}
\langle s_i,s_i \rangle & = & - \cos \left( \pi \right) & = & 1 &\ \mbox{if} \ i=0, \dots,n-1  ,\\
\langle s_i,s_j \rangle & = & - \cos \left( {\pi}/{2} \right) & = & 0 &\ \mbox{if} \ |i-j|>1  ,\\
\langle s_i,s_{i+1} \rangle & = & - \cos \left({\pi}/{3} \right) & = & - {1}/{2} &\ \mbox{if} \ i=1, \dots,n-2  ,\\
\langle s_0,s_1 \rangle & = & - \cos \left( {\pi}/{4} \right) & = & -{\sqrt{2}}/{2}   .& \\
\end{array}
\]
See~\cite{Hu}.

It follows that the automorphisms $\alpha_0, \alpha_1, \ldots, \alpha_{n-1}$ act on
the polynomial algebra $R = \R[X_0, \ldots,X_{n-1}]$ by the formulas
\begin{equation*}
\alpha_0:
\left\{
\begin{array}{ccll}
X_0 & \mapsto & -X_0  , & \\
X_1 & \mapsto & X_1 + \sqrt{2}\,  X_0  , & \\
X_i & \mapsto & X_i & \mbox{if} \  i>1 , 
\end{array}
\right.
\end{equation*}
\begin{equation*}
\alpha_1:
\left\{
\begin{array}{ccll}
X_0 & \mapsto & X_0 + \sqrt{2}\,  X_1  , & \\
X_1 & \mapsto & -X_1   , & \\
X_2 & \mapsto & X_1 + X_2  , & \\
X_i & \mapsto & X_i & \mbox{if} \  i>2  . 
\end{array}
\right.
\end{equation*}
For $1 < j < n$, we have
\begin{equation*}
\alpha_j:
\left\{
\begin{array}{ccll}
X_j & \mapsto & -X_j  , &  \\
X_i & \mapsto & X_i + X_j & \mbox{if} \  i=j-1,j+1  , \\
X_i & \mapsto & X_i & \mbox{if} \  i \neq j-1,j,j+1  .
\end{array}
\right.
\end{equation*}

For future use, we shall need an explicit set of 
algebraically independent homogeneous generators for certain subalgebras~$R^w$
of invariants. Using Chevalley's theorem (see \cite[Chap.~3]{Hu} or \cite[Chap.~31]{TY} for more details), one can verify that 
\[
R^{s_0} = \R [ \sqrt{2}X_0+2X_1, X_{0}^{2}, X_2, \dots, X_{n-1} ] \, ,
\]
\[
R^{s_1s_0s_1} = \R [ X_0, X_1 + X_2 , X_1( \sqrt{2}X_0 + X_1), X_3, \dots,X_{n-1} ] \, ,
\]
\[
R^{s_1} = \R [ 2X_0+\sqrt{2}X_1 , X_1+2X_2,X_1^2, X_3, \dots,X_{n-1}] 
\]
and
\[
R^{s_0s_1s_0} = \R [ X_0 + \sqrt{2}X_2, X_1 , X_0( X_0 + \sqrt{2}X_1 ), X_3, \dots, X_{n-1} ] \, .
\]

\section{Weak categorification of the virtual braid group of type~$B_n$}\label{sec-cat}

\subsection{Rouquier's weak categorification of generalized braid groups}\label{Rougen}

We explain Rouquier's construction following~\cite{R}. 
It works for any generalized braid group~$\B_{\W}$ associated to a finite Coxeter group~$\W$. 

To each braid generator $s_i \in \B_{\W}$ 
we assign the cochain complex $F(s_i)$ of graded $R$-bimodules
\begin{equation}\label{CX sigma}
F(s_i):  0  \longrightarrow  R\{2\}  \xrightarrow{\rb_i}  B_{s_i} 
\longrightarrow  0 \, ,
\end{equation}
where $B_{s_i}$ sits in cohomological degree 0. 
The degree-preserving $R$-bimodule morphism~$\rb_i$ sends any $a\in R$ to $aX_i \ot 1 + a \ot X_i \in B_{s_i}$.

To the inverse~$s_i^{-1}$ of~$s_i$ we assign 
the cochain complex $F(s_i^{-1})$ of graded $R$-bimodules
\begin{equation}\label{CX sigma-1}
F(s_i^{-1}) :  0  \longrightarrow  B_{s_i}\{-2\}  \xrightarrow{\br_i}  
R\{-2\}  \longrightarrow  0 \, ,
\end{equation}
where $B_{s_i}\{-2\}$ sits in cohomological degree 0 and 
the degree-preserving $R$-bimodule morphism $\br_i$ is given by multiplication; in other words, it sends
$a\ot b \in B_{s_i}$ to~$ab \in R$. 

To the unit element $1 \in \B_{\W}$  we assign the trivial complex of graded $R$-bimodules
\begin{equation}\label{CX one}
F(1)  :  0  \longrightarrow  R  \longrightarrow  0 \, ,
\end{equation}
where $R$ sits in cohomological degree~$0$. 
The complex $F(1)$ is obviously a unit for the tensor product of complexes. 

Finally to any word $w = s_{i_1}^{\varepsilon_1} \dots s_{i_k}^{\varepsilon_k}$ 
where $\varepsilon_1, \dots ,\varepsilon_k = \pm 1$, we assign the complex of 
graded $R$-bimodules $F(w)=F(s_{i_1}^{\varepsilon_1})\ot_R \dots \ot_R 
F(s_{i_k}^{\varepsilon_k})$.

In this context, Rouquier established that 
if $w$ and $w'$ are words representing the same element of $\B_{\W}$, 
then $F(w)$ and~$F(w ')$ are homotopy equivalent complexes of $R$-bimodules.
This statement is what we mean by Rouquier's weak categorification of generalized braid groups and what we want, for the type $B$ case, to generalize to virtual braids. But in fact, Rouquier enhanced the above result in~\cite{R}. He also conjectured his categorification to be faithfull while it will not be true for our generalization to virtual braids; see Remarks \ref{pasbizarre}.

\subsection{Weak categorification of~$\VB_{B_n}$}\label{sec-ourcat}

Our aim is to extend the construction of Rouquier detailed in the former section to the virtual braid group~$\VB_{B_n}$ of type $B_n$ defined in Section~\ref{sec-virsymm}. 

The cochain complexes associated to the generators~$s_i^{\pm 1}$ of~$\VB_{B_n}$ are 
the ones defined by~\eqref{CX sigma} and~\eqref{CX sigma-1}. 

To the generator~$z_i$ we assign the complex of graded $R$-bimodules concentrated in degree~$0$
\begin{equation}\label{def-zi}
F(z_i)  :  0  \longrightarrow  R_{s_i}  \longrightarrow  0 \, .
\end{equation}
To any word~$w$ in the generators~$s_i^{\pm 1}$ and~$z_i$ of~$\VB_{B_n}$
we assign the tensor product over~$R$ of the complexes associated to the generators 
involved in the expression of~$w$.

We now state our main result.

\begin{thm}\label{catvirt} 
If $w$ and $w '$ are words representing the same element of~$\VB_{B_n}$, 
then $F(w)$ and~$F(w')$ are homotopy equivalent complexes of $R$-bimodules.
\end{thm}

\begin{proof}
It is enough to check that there are homotopy equivalences between the complexes associated 
to the pair of words appearing in each defining relation of~$\VB_{B_n}$.

The checking of this for Relations~\eqref{relB1}, \eqref{relB2}, \eqref{relB3} is a consequence of
Rouquier's work for generalized braid groups. 

Relations \eqref{relWB1}--\eqref{relWB0} only involve the virtual generators~$z_i$. 
In view of the simple form of the complex~$F(z_i)$,
the isomorphims of the corresponding complexes directly follow from 
the first isomorphism in Lemma~\ref{iso bimod}. 

Similarly, the isomorphism of complexes associated to Relations~\eqref{relmixB1} and~\eqref{relmixB2} 
can be constructed as performed in~\cite{Th1} in the case of the categorification of the virtual braid group of type~$A$.

We are left with the mixed relations~\eqref{relmixB3}--\eqref{relmixB5}.
Let us first deal with Relation~\eqref{relmixB3}. 
We have to prove that the complexes $F(s_0z_1z_0z_1)$ and $F(z_1z_0z_1s_0)$ are isomorphic. 
A simple computation shows that the complex $F( z_1z_0z_1s_0)$ is isomorphic to the following:
\begin{equation*}
0  \longrightarrow  R_{s_1s_0s_1}\{2\} \overset{d}{\longrightarrow}  R_{s_1s_0s_1}  \ot_R B_{s_0} \longrightarrow  0 \, ,
\end{equation*}
where
\[ 
d(a) = a \left( \alpha_1 \alpha_0 \alpha_1 \left(X_0 \right) \ot 1 + 1 \ot X_0 \right) = a \left( X_0 \ot 1 + 1 \ot X_0 \right)
\]
for all $a \in R$.
The complex $F(s_0z_1z_0z_1)$ is isomorphic to
\begin{equation*}
0  \longrightarrow  R_{s_1s_0s_1}\{2\}  \overset{d'}{\longrightarrow}  B_{s_0} \ot_R R_{s_1s_0s_1} \longrightarrow  0 \, ,
\end{equation*}
where 
\[
d'(a) = a \left( X_0 \ot 1 + 1 \ot X_0 \right)
\]
for all $a \in R$.
Lemma~\ref{iso bimod} provides an isomorphism 
\[
\mu : R_{s_1s_0s_1}  \ot_R B_{s_0} \overset{\cong}{\longrightarrow} 
B_{s_1s_0s_1s_0s_1s_0s_1} \ot_R R_{s_1s_0s_1} \, ; 
\]
it is given by $\mu (a\ot b) = a \ot \alpha_1 \alpha_0 \alpha_1(b)$, where $a,b\in R$.
Now by Relation~\eqref{relB3} the latter bimodule is equal to $B_{s_0} \ot_R R_{s_1s_0s_1}$. 
This allows us to build the following isomorphism of complexes 
between $F(z_1z_0z_1s_0)$ and $F(s_0z_1z_0z_1)$:
\[
\xymatrix{
0 \ar[r] &  R_{s_1s_0s_1}\{2\} \ar^(0.45){d}[r] \ar_{\id}[d]  &  R_{s_1s_0s_1}  \ot_R B_{s_0} \ar[r] \ar_{\mu}[d] &  0 \\
0 \ar[r]  &  R_{s_1s_0s_1}\{2\} \ar^(0.45){d'}[r]   &  B_{s_0} \ot_R R_{s_1s_0s_1} \ar[r]  &  0
}
\]
The vertical maps (and similarly their inverse) commute with the differentials. Indeed, for $a\in R_{s_1s_0s_1}$,
\[
\begin{array}{rcl}
\mu \circ d (a) & = & \mu \left(a \left( X_0 \ot 1 + 1 \ot X_0 \right) \right) \\
& = & a \left( X_0 \ot 1 + 1 \ot \alpha_1 \alpha_0 \alpha_1 \left( X_0 \right) \right) \\
& = & a \left( X_0 \ot 1 + 1 \ot X_0 \right) \\
& = & d'(a) \, .
\end{array}
\]

Similar arguments allow us to construct an isomorphism of complexes between 
$F(z_0s_1z_0z_1)$ and~$F(z_1z_0s_1z_0)$,
which proves that~\eqref{relmixB4} is satisfied on the level of complexes.

In order to handle Relation~\eqref{relmixB5}, the first step is to find an isomorphism 
between~$B_{s_0} \ot_R R_{s_1} \ot_R B_{s_0} \ot_R R_{s_1}$ and~$ R_{s_1} \ot_R B_{s_0} \ot_R R_{s_1} \ot_R B_{s_0}$.
Applying Lemma~\ref{iso bimod} again, we first observe that 
\[
B_{s_0} \ot_R R_{s_1} \ot_R B_{s_0} \ot_R R_{s_1} \cong B_{s_0} \ot_R B_{s_1s_0s_1}
\]
and
\[
R_{s_1} \ot_R B_{s_0} \ot_R R_{s_1} \ot_R B_{s_0} \cong B_{s_1s_0s_1} \ot_R B_{s_0} \, .
\]
Then using the generating set of~$R^{w}$ for $w \in \{s_0,s_1s_0s_1\}$ exhibited in Section~\ref{inv}, 
we can make an isomorphism of $R$-bimodules 
\[
\varphi : B_{s_0} \ot B_{s_1s_0s_1} \rightarrow B_{s_1s_0s_1} \ot B_{s_0}
\]
explicit by setting
\begin{equation*}
\begin{array}{rcll}
\varphi \left( 1 \ot 1 \ot 1 \right) & = & 1 \ot 1 \ot 1 \, ,& \\
\varphi \left( 1 \ot X_0 \ot 1 \right) & = & 1 \ot 1 \ot X_0 \, ,& \\
\varphi \left(1 \ot X_1 \ot 1 \right) & = &  -X_2 \ot 1 \ot 1 + 1 \ot 1 \ot \left( X_1 + X_2 \right) \, ,& \\
\varphi \left(1 \ot X_i \ot 1 \right) & = & X_i \ot 1 \ot 1 & \mbox{if} \  i>1 \, . 
\end{array}
\end{equation*}
The map~$\varphi$ is clearly surjective since any element of~$R$ can be written 
as a sum of products of elements of~$R^{s_0}$ and~$R^{s_1s_0s_0}$. 
Let us prove that this isomorphism of $R$-bimodules is well-defined. 
We have to check that
\begin{equation}\label{verif1}
\varphi \left( 1 \ot p \ot 1 \right) = p \ot 1 \ot 1
\end{equation}
for all $p \in R^{s_0}$; and
\begin{equation}\label{verif2}
\varphi \left( 1 \ot p \ot 1 \right) = 1 \ot 1 \ot p
\end{equation}
for all $p \in R^{s_1s_0s_1}$.
It is enough to check \eqref{verif1} (resp.\ \eqref{verif2}) for $p$ equal to the generating elements of $R^{s_0}$ (resp.\ of~$R^{s_1s_0s_1}$). 

For $p = X_0$ (resp.\ $p=X_2$) Equality~\eqref{verif2} (resp.\ Equality~\eqref{verif1}) follows directly from the definition of~$\varphi$. For $p=X_i$ with $i>2$, Equalities~\eqref{verif1} and~\eqref{verif2} follow from the fact that $X_i \in R^{s_0} \cap R^{s_1s_0s_1}$. 

It remains to deal with the elements~$\sqrt{2}X_0 + 2X_1$ and~$X_{0}^{2}$ of~$R^{s_0}$,
and with the elements $X_1 +X_2$  and $X_1 \left( \sqrt{2}X_0 +X_1 \right)$ of~$R^{s_1s_0s_1}$. 
Equality~\eqref{verif1} is obvious for~$X_{0}^{2}$ since $X_{0}^{2} \in R^{s_0} \cap R^{s_1s_0s_1}$. 
By definition,
\[
\varphi \bigl( 1 \ot ( \sqrt{2}X_0 + 2X_1 ) \ot 1 \bigr) 
= -2X_2 \ot 1 \ot 1 + 1 \ot 1 \ot ( \sqrt{2}X_0 +2X_1 +2X_2 ) .
\]
Now, $\sqrt{2}X_0 +2X_1 +2X_2$ belongs to $R^{s_0} \cap R^{s_1s_0s_1}$, so we get the expected equality
\[
\varphi \bigl( 1 \ot ( \sqrt{2}X_0 + 2X_1 ) \ot 1 \bigr) =    ( \sqrt{2}X_0 + 2X_1 ) \ot 1  \ot 1 \, .
\]
For $X_1 +X_2$, Equality~\eqref{verif2} follows directly from the definition of~$\varphi$; we have
\begin{eqnarray*}
\varphi \bigl(1 \ot ( X_1 +X_2 ) \ot 1 \bigr)
& = & -X_2 \ot 1 \ot 1 + 1 \ot 1 \ot (X_1 +X_2 ) + X_2 \ot 1 \ot 1 \\
& = & 1 \ot 1 \ot \left(X_1 +X_2 \right) \, .
\end{eqnarray*}

The last element $p \in R^{s_1s_0s_1}$ for which we have to check Equality~\eqref{verif2} is $p=X_1 \left( \sqrt{2}X_0 +X_1 \right)$. By definition,
\[
\begin{array}{rcl}
\varphi \bigl( 1 \ot X_1 ( \sqrt{2}X_0 +X_1 ) \ot 1\bigr) 
& = &
-X_2 \ot 1 \ot \sqrt{2}X_0  \\
&  &+ 1 \ot 1 \ot \sqrt{2}X_0 ( X_1 + X_2 )+ X_{2}^{2} \ot 1 \ot 1   \\
&  & -2 X_2 \ot 1 \ot (X_1 + X_2 ) \\
&  & + 1 \ot 1 \ot (X_1 + X_2 )^2 \\
& = & 1 \ot 1 \ot \bigl( \sqrt{2}X_0 ( X_1 + X_2 ) + (X_1 + X_2 )^2 \bigr) \\
&   & - X_2 \ot 1 \ot ( \sqrt{2}X_0 +2X_1 +2X_2 ) \\
&  & + X_{2}^{2} \ot 1 \ot 1 \, .
\end{array}
\]
Now remark that $X_{2}^{2}$ can be written as follows:
\begin{multline*}
X_{2}^{2} = 
\bigl( X_1 ( \sqrt{2}X_0 + X_1 ) - (X_1 + X_2 )^2 - \sqrt{2}X_0 ( X_1 + X_2 ) \bigr) \\
+ ( \sqrt{2} X_0 + 2X_1 + 2X_2 ) X_2 \, . 
\end{multline*}
One can check that both 
\[
X_1 ( \sqrt{2}X_0 + X_1 ) - (X_1 + X_2 )^2 - \sqrt{2}X_0 ( X_1 + X_2 )
\] 
and $\sqrt{2} X_0 + 2X_1 + 2X_2 $ belong to~$R^{s_0} \cap R^{s_1s_0s_1}$. 
So replacing~$X_{2}^{2}$ by its latter expression in the expression of $\varphi \bigl( 1 \ot X_1 ( \sqrt{2}X_0 +X_1) \ot 1 \bigr)$, 
we obtain what we expected, namely
\[
\varphi \bigl( 1 \ot X_1 ( \sqrt{2}X_0 +X_1) \ot 1 \bigr)= 1 \ot 1 \ot X_1 ( \sqrt{2}X_0 +X_1)  \, .
\]

The existence of the bimodule isomorphism~$\varphi$ leads to an isomorphism between the complexes 
$F(s_0z_1s_0z_1)$ and $F(z_1s_0z_1s_0)$. In fact,
$F(s_0z_1s_0z_1)$ is isomorphic to the complex
\[
\xymatrix{
    &   & B_{s_0}\{2\} \ar^{d^{-1}_{1}}[dr]  &  & \\
  0 \ar[r] &  R\{4\} \ar^{d^{-2}_{1}}[ur] 
\ar_{d^{-2}_{2}}[dr] &   & B_{s_0} \ot_R B_{s_1s_0s_1} \ar[r] & 0  \\
  &   & B_{s_1s_0s_1}\{2\} \ar_{d^{-1}_{2}}[ur]  &  & & \\
}
\]
whose differentials are obtained by composing the ones of $F(s_0z_1s_0z_1)$  
with the isomorphism of Lemma~\ref{iso bimod}. More precisely,
\[
\begin{array}{rcl}
d^{-2}_{1} (a) & = & a \left( X_0 \ot 1 +1 \ot X_0 \right) ,\\
d^{-2}_{2} (a) & = & -a \left(\alpha_1 \left(X_0 \right) \ot 1 + 1 \ot \alpha_1 \left(X_0 \right) \right) \\
 & = & -a \left( \left(X_0 + \sqrt{2}X_1 \right) \ot 1 + 1 \ot  \left(X_0 + \sqrt{2}X_1 \right) \right), \\
d^{-1}_{1} (a \ot b) & = & a \left( 1 \ot \alpha_1 \left(X_0 \right) \ot 1 + 1 \ot 1 \ot \alpha_1 \left(X_0 \right) \right) b\\
& = & a \left( 1 \ot \left(X_0 + \sqrt{2}X_1 \right) \ot 1 + 1 \ot 1 \ot \left(X_0 + \sqrt{2}X_1 \right) \right) b,\\
d^{-1}_{2} (a \ot b) & = & a \left( X_0 \ot 1 \ot 1 +1 \ot X_0 \ot 1 \right) b 
\end{array}
\]
for all $a,b,c \in R$.
Similarly, the complex $F(z_1s_0z_1s_0)$ is isomorphic to
\[
\xymatrix{
    &   & B_{s_0}\{2\} \ar^{d'^{-1}_{1}}[dr]  & &  \\
 0 \ar[r] &  R\{4\} \ar^{d'^{-2}_{1}}[ur] 
\ar_{d'^{-2}_{2}}[dr] &   &  B_{s_1s_0s_1} \ot_R B_{s_0} \ar[r] & 0  \\
  &   & B_{s_1s_0s_1}\{2\} \ar_{d'^{-1}_{2}}[ur]  &  & &  \\
}
\]
whose differentials are obtained by composing the ones of~$F(z_1s_0z_1s_0)$ with the isomorphism of Lemma~\ref{iso bimod}, namely
\[
\begin{array}{rcl}
d'^{-2}_{1} (a) & = & -a \left( X_0 \ot 1 +1 \ot X_0 \right), \\
d'^{-2}_{2} (a) & = & a \left(\alpha_1 \left(X_0 \right) \ot 1 + 1 \ot \alpha_1 \left(X_0 \right) \right) \\
 & = & a \left( \left(X_0 + \sqrt{2}X_1 \right) \ot 1 + 1 \ot  \left(X_0 + \sqrt{2}X_1 \right) \right), \\
d'^{-1}_{1} (a \ot b) & = & a \left(  \alpha_1 \left(X_0 \right) \ot 1 \ot 1 + 1 \ot \alpha_1 \left(X_0 \right) \ot 1 \right) b\\
& = & a \left( \left(X_0 + \sqrt{2}X_1 \right) \ot 1 \ot 1 + 1 \ot \left(X_0 + \sqrt{2}X_1 \right) \ot 1 \right) b,\\
d'^{-1}_{2} (a \ot b) & = & a \left( 1 \ot X_0 \ot 1 +1 \ot 1 \ot X_0 \right) b 
\end{array}
\]
for all $a,b,c \in R$. We define an isomorphism between these two complexes using the identity map (up to a sign) for all factors outside of cohomological degree~$0$ and using~$\varphi$ in degree~$0$. This can be summarized by the following commutative diagram: 
\[
\xymatrix{
    &   & B_{s_0}\{2\} \ar@/l1cm/_(0.6){\id}[dddd] \ar^{d^{-1}_{1}}[dr]  &  & \\
  0 \ar[r] &  R\{4\} \ar_{-\id}[dddd] \ar^{d^{-2}_{1}}[ur] \ar_{d^{-2}_{2}}[dr] &   & B_{s_0} \ot_R B_{s_1s_0s_1} \ar[r] \ar_{\varphi}[dddd] & 0  \\
  &   & B_{s_1s_0s_1}\{2\} \ar_{d^{-1}_{2}}[ur] \ar@/l1cm/_(0.4){\id}[dddd] &  & & \\
 &  &  &  &  \\
    &   & B_{s_0}\{2\} \ar^{d'^{-1}_{1}}[dr]   & &  \\
  0 \ar[r] &  R\{4\}  \ar^{d'^{-2}_{1}}[ur] \ar_{d'^{-2}_{2}}[dr] &   &  B_{s_1s_0s_1} \ot_R B_{s_0} \ar[r]  & 0  \\
  &   & B_{s_1s_0s_1}\{2\} \ar_{d'^{-1}_{2}}[ur]   &  & &  \\
}
\]
This morphism of complexes is well-defined: the identities commute with the differentials 
and for all $a, b\in R$,
\[
\begin{array}{rcl}
\varphi \circ d_{1}^{-1} (a\ot b) & = & \varphi \left( a \left( 1 \ot \left(X_0 + \sqrt{2}X_1 \right) \ot 1 + 1 \ot 1 \ot \left(X_0 + \sqrt{2}X_1 \right) \right) b \right) \\
& = & a \left( \left(X_0 + \sqrt{2}X_1 \right) \ot 1 \ot 1 + 1 \ot 1 \ot \left(X_0 + \sqrt{2}X_1 \right) \right) b \\
& = & d'^{-1}_{1} (a\ot b) 
\end{array}
\]
since $\left(X_0 + \sqrt{2}X_1 \right) \in R^{s_0}$. We also have
\[
\begin{array}{rcl}
\varphi \circ d^{-1}_{2} (a \ot b) & = & \varphi \left( a \left( X_0 \ot 1 \ot 1 +1 \ot X_0 \ot 1 \right) b \right) \\
& = & a \left( X_0 \ot 1 \ot 1 +1 \ot 1 \ot X_0 \right) b \\
& = & d'^{-1}_{2} (a \ot b) 
\end{array}
\]
since $ X_0 \in R^{s_1 s_0 s_1}$. This completes the proof.
\end{proof}

\begin{rems}\label{pasbizarre} \hfill\\
$\bullet$ It is important to work in the homotopy category rather than in the derived category since the complexes $F(s_i)$ and $F(s_{i}^{-1})$ are both quasi-isomorphic to $F(z_i)$ while they are not homotopy equivalent to it. 

\noindent $\bullet$ We observed in Proposition~\ref{bizarre} that quite unexpectedly the relation 
\[
z_0s_1z_0s_1 = s_1z_0s_1z_0
\]
does not hold in~$\VB_{B_n}$.
Nevertheless the two complexes $F(z_0s_1z_0s_1)$ and $F(s_1z_0s_1z_0)$~are isomorphic.

The existence of an isomorphism between these complexes follows 
from the isomorphim of $R$-bimodules 
$$\psi:B_{s_1} \ot_R B_{s_0s_1s_0} \rightarrow B_{s_0s_1s_0} \ot_R B_{s_1}.$$
The isomorphism $\psi$ can be expressed in a similar way as the isomorphism $\varphi: B_{s_0} \ot B_{s_1s_0s_1} \rightarrow B_{s_1s_0s_1} \ot B_{s_0}$. Then, just as in the last part of the proof of theorem \ref{catvirt}, this enable us to construct an isomorphism between the complexes $F(z_0s_1z_0s_1)$ and $F(s_1z_0s_1z_0)$.

\noindent $\bullet$ As in the type~$A$ case (treated in~\cite{Th1}), 
the complexes $F(s_iz_i)$ and $F(z_is_i)$ are isomorphic,
which contrasts with the fact that the relation $s_iz_i = z_is_i$ does not hold in~$\VB_{B_n}$. 

\noindent $\bullet$ Although we do not manage to state here that $j:\VB_{B_n} \rightarrow \VB_{2n}$ is injective, it would be interesting to know at least if, for any $a\in \ker(j)$, the complexes $F(a)$ and $F(1)$ are homotopy equivalent.
\end{rems}

\vspace{-0.2cm}

\section*{Acknowledgements}

I thank my thesis advisor C.~Kassel for his guidance and his help with writing up this paper.

\vspace{-0.3cm}


\bibliographystyle{alpha}

\bibliography{biblio4}

\end{document}